\documentclass[preprint,12pt]{elsarticle}

\usepackage{amssymb}
\usepackage{amsthm}

\usepackage{multirow,multicol}
\usepackage{booktabs}
\usepackage{amsmath, amsthm, amscd, amsfonts, amssymb, graphicx, color}
\usepackage[bookmarksnumbered, plainpages, backref]{hyperref}
\usepackage [latin1]{inputenc}

\usepackage{graphics}
\usepackage{epsfig}
\usepackage{graphicx}
\usepackage{epstopdf}
\usepackage{amsthm}

\theoremstyle{plain}
\newtheorem{theorem}{Theorem}[section]

\newtheorem{lemma}[theorem]{Lemma}
\newtheorem{remark}[theorem]{Remark}

\usepackage{float}%¹Ì¶¨Í¼Æ¬ÔÚÖ¸¶¨µÄÎ»ÖÃ

\journal{XXX}

\begin{document}

\begin{frontmatter}

\title{Numerical integrations of stochastic contact Hamiltonian systems via stochastic contact Hamilton-Jacobi equation}

\author[mymainaddress]{Qingyi Zhan \corref{mycorrespondingauthor}}
\ead{zhan2017@fafu.edu.cn}
\address[mymainaddress]{College  of Computer and Information Science,\\
                         Fujian Agriculture and Forestry  University,Fuzhou,350002,China}

\author[mysecondaryaddress]{Jinqiao Duan }
%\ead{duan@gbu.edu.cn}
\address[mysecondaryaddress]{College of Science,Great Bay University, Dongguan, Guangdong, 523000, China;\\Center for Mathematical Sciences, Huazhong University of Science and Technology, Wuhan, Hubei, 430074, China}

\author[mythirdaddress]{Xiaofan Li}
%\ead{lix@iit.edu}
\address[mythirdaddress]{Department of Applied Mathematics,\\
                         Illinois Institute of Technology,Chicago,IL,60616,USA}

\cortext[mycorrespondingauthor]{Corresponding author}

\author[myfifthaddress]{Lijin Wang}
%\ead{ljwang@ucas.ac.cn}
\address[myfifthaddress]{ School of Mathematical Sciences,\\
   University of Chinese Academy of Sciences,
    Beijing,100049,China}

\begin{abstract}
  Stochastic contact Hamiltonian systems are a class of important mathematical models, which can describe the dissipative properties  with odd dimensions  in
the stochastic environment. In this article, we investigate the numerical dynamics of the stochastic contact Hamiltonian systems  via structure-preserving methods. The contact structure-preserving schemes are constructed by the stochastic contact Hamilton-Jacobi equation. A general numerical approximation method of the stochastic contact Hamilton-Jacobi equation is devised, and the convergent order theorem  is provided, too.
 Numerical tests are shown to confirm the theoretical results and the usability of proposed approach.

\end{abstract}


\end{frontmatter}

\section{Introduction}

It should be noticed that stochastic contact Hamiltonian systems are related to  many fields, such as the synchrotron oscillations of particles in storage rings, especially in low-dimensional topology case\cite{Arnold, Bravetti, Feng, Geiges}. In 1980s, Misawa investigated the Hamiltonian formalism of Nelson's stochastic mechanics  via a stochastic Hamilton-Jacobi theory. \cite{Misawa}. Nowadays, stochastic contact Hamiltonian  differential equations have been increased more and more interesting \cite{Wei,Duan, Leon, Zhan5}. It is most due to the fact that these systems come from an important class of stochastic differential equations(SDEs), which contain contact structure  in the odd dimension. However, it was probably until now that there is no general description of stochastic contact Hamiltonian  differential equations via contact Hamilton-Jacobi equation.

As we know, it is difficult to solve SDEs analytically in general, then the reliable numerical methods are needed in many fields, such as numerical simulation. Contact Hamilton-Jacobi theory reveals the relationship between  stochastic contact Hamiltonian  differential equations and the contact Hamilton-Jacobi partial differential equations. That is to say, the solution of the former can generate the phase flow of the latter.

Therefore, the dynamics of the flows of stochastic contact Hamiltonian  differential equations can be deduced by the contact Hamilton-Jacobi equations, which  possess geometric character that similarly parallel their symplectic counterparts in this aspect.
Contact Hamilton-Jacobi equation has been an important approach of investigating the dynamics and evolutions of stochastic contact Hamiltonian  differential equations\cite{Georgieva}. They have relationship with stochastic contact Hamilton-Jacobi  theory. There is a difficulty in numerical simulation to Hamiltonian  differential equations, which is the formulation of the appropriate discrete of contact Hamilton-Jacobi equation. \cite{Bravetti02,Leon, Leon2}.

Via  the contact Hamilton-Jacobi equation, in this article we investigate the numerical dynamics for stochastic contact Hamiltonian  differential equations. There are two factors for this work. Firstly, the associated contact Hamilton-Jacobi equation for contact Hamiltonian  differential equations are paid more and more attention\cite{Georgieva, Liu}. Therefore we expect to expand these results to the stochastic contact Hamiltonian  differential equations. Secondly, there are also many existed contributions in the numerical analysis and numerical methods of SDEs,  which can be seen in \cite{Wang, Hong, X. Wang} and \cite{Milstein01, Zhan4, Zhan5}. All these provide strong foundations for our work. According to the information I have, systematic investigation of numerical dynamics of stochastic contact Hamiltonian  differential equations via stochastic contact Hamilton-Jacobi equation  is not existed in the literatures.

In this paper, we attempt to show that stochastic contact Hamiltonian  differential equations can be obtained by its associated contact Hamilton-Jacobi equation, whose dynamics is preserved almost surely. We also argue that contact Hamilton-Jacobi equation plays a similar role, which is the same as their symplectic counterparts for stochastic dissipative systems.
 In the numerical experiments, we try to compare the numerical dynamical behaviors by contact Hamilton-Jacobi equation in many aspects.

The paper is organized as follows. Section 2 provides some preliminaries for the following parts. The  theoretical results on stochastic contact Hamilton-Jacobi equation are summarized in Section 3. Section 4 shows the construction of contact structure-preserving methods via the discrete version of contact Hamilton-Jacobi equation. Section 5 is about the convergence analysis of a class of order 1.0 contact structure-preserving scheme. In section 6, we investigate some numerical experiments to testify our theoretical results and the validity of our approach. Section 7 is the conclusion.

\section{Preliminaries}
Let $(\Omega, \mathcal{F}, \{\mathcal{F}_t\}_{t\geq 0}, \mathbb{P})$ be a given probability space.
Consider the $d=(2n+1)$-dimensional Stratonovich stochastic  Hamiltonian  systems as follows,
  \begin{equation}\label{2.1}
\left \{
\begin{aligned}
 dQ(t)&=\frac{\partial H_0}{\partial P}(Q(t),P(t),S,t)dt
  +\sum_{k=1}^m \frac{\partial H_k}{\partial P}(Q(t),P(t),S,t)\circ dW^k(t),\\
 dP(t)&=-\Big[\frac{\partial H_0}{\partial Q}(Q(t),P(t),S,t)+P(t)\frac{\partial H_0}{\partial S}(Q(t),P(t),S,t)\Big ]dt\\
 &-\sum_{k=1}^m \Big(\frac{\partial H_k}{\partial Q}(Q(t),P(t),S,t)+P(t)\frac{\partial H_k}{\partial S}(Q(t),P(t),S,t)\Big)\circ dW^k(t),\\
 dS&=\Big[P(t)\frac{\partial H_0}{\partial P}(Q(t),P(t),S,t)-H_0(Q(t),P(t),S,t) \Big ] dt\\
 &+\sum_{k=1}^m \Big[P(t)\frac{\partial H_k}{\partial P}(Q(t),P(t),S,t)-H_k(Q(t),P(t),S,t)\Big]\circ dW^k(t),
 \end{aligned}
\right.
\end{equation}
where $(Q(t),P(t),S)_{t=0}=(q,p,s)=(q^1,q^2,...,q^n,p^1,p^2,...,p^n,s)$ is the initial value. The coefficient functions $H_k(k=0,1,...,m)$ satisfy conditions, which guarantee the existence and uniqueness of the solutions $(Q(t),P(t),S)$ to the systems  $(\ref{2.1})$. $W(t):=(W^1(t), W^2(t),...,W^m(t))$ is a m-dimensional independent standard Brownian motion \cite{Wei}. The norm of random variable and
the norm of random matrices is defined as same as in \cite{Zhan5}, which is referred in \cite{Golub}.

\section{Stochastic contact Hamilton-Jacobi Theory}

Similar to the symplectic Hamiltonian systems, stochastic contact Hamilton-Jacobi theory is a powerful tool to connect the stochastic Hamiltonian-Jacobi partial differential equations and the stochastic contact Hamiltonian  systems. It is enabled to re-express stochastic contact Hamiltonian's systems in the light of a single stochastic partial differential equation(SPDE), that is to say,  the solution of the latter can generate the phase flow of the former. \cite{Milstein01}-\cite{Misawa}

Let $\varphi_t(q,p,s):(q,p,s)\mapsto (Q(t),P(t),S)$ be the flow  of a  stochastic contact Hamiltonian  systems $(\ref{2.1})$ with initial value $(q,p,s)$ for any time $t\in [0,T]$.

As we know, the contact form is invariant by a conformal factor, that is,
 $$dS-PdQ=\lambda_t(ds-pdq),$$
 where
 \begin{equation}\label{3.1}
\lambda_t=\exp\Big[\int^t_{0} \frac{\partial H_0}{\partial S} d\tau-\sum_{k=1}^m\int^t_{0}\frac{\partial H_k}{\partial S}\circ dW^k\Big].
\nonumber
\end{equation}

In this paper, notice  that the variable $S$ is not only the generating function of the contact transformation, but also a generalization of Hamilton's principal function, which satisfies the third equation of the systems $(\ref{2.1})$. We will prove that it satisfies the following contact version of the stochastic contact Hamilton-Jacobi equation. Now we first assume that the variable $S$  depends on the independent  variables $Q(t)$ and $t$, and the nonadditive constant $C$, that is $S(Q,C,t)$. Then we obtain   the following theorem.
\begin{theorem}
Let  $S(Q,C,t)$ satisfy the stochastic contact Hamilton-Jacobi equation in the form of
\begin{equation}\label{3.2}
dS+H_0(Q,\frac{\partial S}{\partial Q},S,t)dt+\sum_{k=1}^mH_k(Q,\frac{\partial S}{\partial Q},S,t)\circ dW^k=0,
\end{equation}
where $C=(C_1,C_2,...,C_n)$ is constant, and suppose
$$\Big|\frac{\partial ^2S}{\partial Q\partial C}\Big|\neq 0.$$

If there exists a stopping time $\tau>t_0$ a.s.,
then for almost every $\omega\in \Omega$,
for each $t\in[t_0,\tau)$, the mapping $(q,p,s)\rightarrow (Q(t),P(t),S(Q,C,t))$ given by $(\ref{3.3})$
 \begin{equation}\label{3.3}
 \left \{
 \begin{aligned}
 &P(t)=\frac{\partial S}{\partial Q}(Q(t),C,t), b=\frac{\partial S}{\partial C}(Q(t),C,t),\\
 &\frac{db}{dt}=-\Big[\frac{\partial H_0 }{\partial S}(Q(t),\frac{\partial S}{\partial Q},S,t)+\sum_{k=1}^m\frac{\partial  H_k} {\partial S}(Q(t),\frac{\partial S}{\partial Q},S,t)\circ \dot{W}^k\Big]b.
 \end{aligned}
 \right.
\end{equation}
is the phase flow of the stochastic contact Hamiltonian systems $(\ref{2.1})$.
\end{theorem}

\begin{proof}
First, it follows from the assumption that $S(Q,C,t)$ is the solution of $(\ref{3.2})$, where $C$ is a constant and also assume that
$$\Big|\frac{\partial ^2S}{\partial Q\partial C}\Big|\neq 0.$$
It follows from $(\ref{3.3})$, the Hamilton-Jacobi  equation $(\ref{3.2})$ can be rewritten as
\begin{equation}\label{3.4}
dS+H_0(Q,P,S,t)dt+\sum_{k=1}^mH_k(Q,P,S,t)\circ dW^k=0.
\end{equation}
According to the second equation of $(\ref{3.3})$, we can have
\begin{equation}\label{3.6}
\frac{db}{dt}=\frac{\partial^2 S}{\partial Q \partial C}\frac{dQ}{dt}+\frac{\partial^2 S}{\partial t \partial C}.
\end{equation}
Deriving the equation $(\ref{3.4})$ with respect to $C$, it can be rewritten as
\begin{equation}\label{3.7}
\begin{aligned}
&\frac{\partial ^2S}{\partial t\partial C}+\Big[\frac{\partial H_0}{\partial S}(Q(t),P(t),S,t)+\sum_{k=1}^m\frac{\partial H_k} {\partial S}(Q(t),P(t),S,t)\circ \dot{W}^k\Big]b
\\
&+\Big[\frac{\partial H_0 }{\partial P}(Q(t),P(t),S,t)+\sum_{k=1}^m\frac{\partial H_k }{\partial P}(Q(t),P(t),S,t)\circ \dot{W}^k\Big]\frac{\partial ^2S}{\partial C\partial Q}=0,
\end{aligned}
\end{equation}
where
\begin{equation}
\frac{\partial P(t)}{\partial  C}=\frac{\partial^2 S}{\partial Q \partial c}.
\nonumber
\end{equation}
Then submit $(\ref{3.7})$ into  $(\ref{3.6})$, we obtain
\begin{equation}\label{3.8}
\begin{aligned}
\frac{db}{dt}=\frac{\partial^2 S}{\partial Q \partial C}\Big[\frac{dQ}{dt}-\frac{\partial H_0 }{\partial P}(Q(t),P(t),S,t)-\sum_{k=1}^m\frac{\partial H_k}{\partial P}(Q(t),P(t),S,t)\circ \dot{W}^k\Big]\\
-\Big[\frac{\partial H_0}{\partial S}(Q(t),P(t),S,t)+\sum_{k=1}^m\frac{\partial  H_k} {\partial S}(Q(t),P(t),S,t)\circ \dot{W}^k\Big]b.
\end{aligned}
\end{equation}
According to the third equality of $(\ref{3.3})$, we have
$$\frac{\partial^2 S}{\partial Q \partial C}\Big[\frac{dQ}{dt}-\frac{\partial H_0 }{\partial P}(Q(t),P(t),S,t)-\sum_{k=1}^m\frac{\partial H_k}{\partial P}(Q(t),P(t),S,t)\circ \dot{W}^k\Big]=0,$$
that is,
\begin{equation}\label{3.80}
\begin{aligned}
dQ(t)&=-\frac{\partial H_0}{\partial P}(Q(t),P(t),S,t)dt-\sum_{k=1}^m\frac{\partial H_k}{\partial P}(Q(t),P(t),S,t)\circ dW^k,\\
\end{aligned}
\end{equation}
 which is the same as the first  equation of $(\ref{2.1})$.

Secondly, by the expression of $P(t)$, we can get
\begin{equation}\label{3.9}
\frac{dP(t)}{dt}=\frac{\partial^2 S}{\partial Q \partial Q}\frac{dQ(t)}{dt}+\frac{\partial^2 S}{\partial Q \partial t}
\end{equation}
Deriving the equation $(\ref{3.4})$ with respect with $Q$, we have
\begin{equation}\label{3.10}
\begin{aligned}
\frac{\partial ^2S}{\partial t\partial Q}&=-\Big[\frac{\partial H_0 }{\partial Q}(Q(t),P(t),S,t)+\sum_{k=1}^m\frac{\partial H_k} {\partial Q}(Q(t),P(t),S,t)\circ \dot{W}^k\Big]\\
&-\Big[\frac{\partial H_0 }{\partial S}(Q(t),P(t),S,t)+\sum_{k=1}^m\frac{\partial H_k } {\partial S}(Q(t),P(t),S,t)\circ \dot{W}^k\Big]P(t) \\
&-\Big[\frac{\partial  H_0}{\partial P}(Q(t),P(t),S,t)+\sum_{k=1}^m\frac{\partial H_k}{\partial P}(Q(t),P(t),S,t)\circ \dot{W}^k\Big]\frac{\partial ^2S(t)}{\partial Q\partial Q}.
\end{aligned}
\end{equation}
Substituting  $(\ref{3.10})$ into $(\ref{3.9})$, we obtain
\begin{equation}\label{3.11}
\begin{aligned}
\frac{dP(t)}{dt}&=-\Big[\frac{\partial H_0}{\partial Q}(Q(t),P(t),S,t)+\sum_{k=1}^m\frac{\partial H_k} {\partial Q}(Q(t),P(t),S,t)\circ \dot{W}^k\Big]\\
&-\Big[\frac{\partial H_0 }{\partial S}(Q(t),P(t),S,t)+\sum_{k=1}^m\frac{\partial H_k } {\partial S}(Q(t),P(t),S,t)\circ \dot{W}^k\Big]P(t) \\
&+\Big[\frac{dQ(t)}{dt}-\frac{\partial H_0 }{\partial P}(Q(t),P(t),S,t)-\sum_{k=1}^m\frac{\partial H_k }{\partial P}(Q(t),P(t),S,t)\circ \dot{W}^k\Big]\frac{\partial ^2S}{\partial Q\partial Q}.
\end{aligned}
\end{equation}
Then it follows from $(\ref{3.80})$  that we have
\begin{equation}\label{3.13}
\begin{aligned}
dP(t)&=-\Big[\frac{\partial H_0}{\partial Q}(Q(t),P(t),S,t)+P(t)\frac{\partial H_0}{\partial S}(Q(t),P(t),S,t)\Big]dt\\
&-\sum_{k=1}^m\Big[\frac{\partial  H_k} {\partial Q}(Q(t),P(t),S,t)+P(t)\frac{\partial H_k} {\partial S}(Q(t),P(t),S,t)\Big]\circ dW^k,
\end{aligned}
\end{equation}
 which is the same as the second equation of $(\ref{2.1})$.

 Lastly, by the fact that $C$ is a constant, we have
 \begin{equation}\label{3.14}
\begin{aligned}
dS&=\frac{\partial S}{\partial Q}dQ+\frac{\partial S}{\partial t}dt\\
&=\Big[P(t)\frac{\partial H_0}{\partial P}(Q(t),P(t),S,t)-H_0(Q(t),P(t),S,t) \Big ] dt\\
 &+\sum_{k=1}^m \Big[P(t)\frac{\partial H_k}{\partial P}(Q(t),P(t),S,t)-H_k(Q(t),P(t),S,t)\Big]\circ dW^k(t),\\
\end{aligned}
\end{equation}
 which is coincide with the last equation of $(\ref{2.1})$.

We complete the proof of Theorem 3.1.
\end{proof}

\begin{remark}
In Theorem 3.1, if $H_k(k=0,1,...,m)$ does not depend on $Q$ explicitly, the constant $C$ can be the initial value $p$. That is, in this case,  the generating function is $S(Q,p,t)$.  Theorem 3.1 also  provides a powerful tool to construct contact sreucture-preserving methods. This is due to the fact that Theorem 3.1 can guarantee the numerical scheme is contact scheme.
\end{remark}

\section{Numerical scheme via the stochastic contact Hamilton-Jacobi equation}
\subsection{The construction of the contact  scheme }
It can be seen from Section 3 that the stochastic contact Hamilton-Jacobi equation $(\ref{3.2})$ is re-formulation of the dynamical equation on the basis of partial differential equation. This is maybe due to the fact that its characteristic curves are equal to  the  stochastic contact Hamiltonian  systems $(\ref{2.1})$.

For stochastic symplectic case, L. Wang et.al \cite{Hong} construct symplectic schemes via generating functions, whose key idea is to  approximate the solution of the stochastic Hamilton-Jacobi PDE and to construct stochastic symplectic scheme with the help of this numerical approximated solution. On the base of this idea,
now we meet with the problem that how to solve $S$ from the stochastic contact Hamiton-Jacobi equation $(\ref{3.2})$.

We follow the methods and the notations in \cite{Deng}, and we can assume that the generating function $S(Q,C,t)$ has the expression in the form of
 \begin{equation}\label{4.1}
 \begin{aligned}
S(Q,C,t)&=G_{(1)}(Q)J_{(1)}(t)\\
&+G_{(1,1)}(Q)J_{(1,1)}(t)+G_{(0)}(Q)J_{(0)}(t)\\
&+G_{(0,1)}(Q)J_{(0,1)}(t)+G_{(1,0)}(Q)J_{(1,0)}(t)+...\\
&=\sum_\alpha G_{\alpha}J_{\alpha}(t),
\end{aligned}
\end{equation}
where $\alpha=(j_1,j_2,...,j_l),j_i\in\{0,1,...,m\},i=1,...,l$ is a multi-index of length $l$,
and $J_{\alpha}$ is the multiple Stratonovich integral

 \begin{equation}\label{4.2}
J_{\alpha}(t)=\int_0^t\int_0^{u_l}\cdots\int_0^{u_2}\circ dW^{j_1}(u_1)\cdots\circ dW^{j_l-1}(u_{l-1})\circ dW^{j_l}(u_l),
\end{equation}
and $du$ is denoted by $dW_0(u)$. The multiple Stratonovich integral $I_{\alpha}$(t) is denoted by

 \begin{equation}\label{4.3}
I_{\alpha}(t)=\int_0^t\int_0^{u_l}\cdots\int_0^{u_2} dW^{j_1}(u_1)\cdots dW^{j_l-1}(u_{l-1}) dW^{j_l}(u_l).
\end{equation}
That is,
$$J_{(0)}(t)=\int_0^t\circ du_1=t,J_{(0,1)}=\int_0^t\int_0^{u_2}\circ du_1\circ dW(u_2)=\int_0^tu_2\circ dW(u_2),$$
$$J_{(1)}(t)=\int_0^t\circ dW(u_1)=W(t),J_{(1,0)}(t)=\int_0^t\int_0^{u_2}\circ dW(u_1)\circ du_2=\int_0^tW(u_2)\circ du_2,$$
and
$$J_{(1,1)}(t)=\int_0^t\int_0^{u_2}\circ dW(u_1)\circ dW(u_2)=\int_0^tW(u_2)\circ dW(u_2).$$
And the expression of $G_{(0)}(Q)t$ can be assumed to have the form of
$$G_{(0)}(Q)t=C^{(0)}_0(t)+C^{(0)}_1(t)Q+C^{(0)}_2(t)Q^2+,...,+C^{(0)}_k(t)Q^k,$$
where $k$ is the highest order of $Q$ in the Hamiltonian $\mathcal{H}$.

Motivated by the work of \cite{Feng}, differentiating two sides of $(\ref{4.1})$ with respect to $t$, we can get
 \begin{equation}\label{4.4}
\begin{aligned}
\frac{\partial S}{\partial t}
&=\bigg[\sum_{\alpha\backslash (0)} G_{\alpha}J'_{\alpha}(t)\bigg]+\bigg[\frac{dC^{(0)}_0(t)}{dt}+\frac{dC^{(0)}_1(t)}{dt}Q+,...,+\frac{dC^{(0)}_k(t)}{dt}Q^k\bigg]\\
&=A+B\dot{W},
\end{aligned}
\end{equation}
where
  \begin{equation}\label{4.5}
 \left \{
\begin{aligned}
A=&\frac{dC^{(0)}_0(t)}{dt}+\frac{dC^{(0)}_1(t)}{dt}Q+,...,+\frac{dC^{(0)}_k(t)}{dt}Q^k+G_{(1,0)}(Q) W(t)+...\\
B=&G_{(1)}(Q)+G_{(1,1)}(Q)W(t)+G_{(0,1)}(Q)t+....
\end{aligned}
\right.
\end{equation}
Comparing  $(\ref{4.4})$  with $(\ref{3.2})$ about the corresponding same terms, we get
 \begin{equation}\label{4.6}
 \left \{
\begin{aligned}
A=-H_0\\
B=-H_1.
\end{aligned}
\right.
\end{equation}
Then the numerical approximation of $S$, i.e., $\bar{S}$ can be obtained by truncating the series $(\ref{4.5})$  from the expressions of $H_0$ and $H_1$. Once  $\bar{S}$  has been obtained, we can utilize the relationship $(\ref{3.3})$ and get the numerical approximation of $P$, i.e.,  $\bar{P}=\frac{\partial \bar{S}}{\partial Q }$. Finally, we can obtain the expression of the numerical approximation of $Q$, i.e., $\bar{Q}$ from the last two equations in  $(\ref{3.3})$.

\begin{theorem}
By the numerical implementation of $(Q,P,S)$, the numerical scheme $(\bar{Q}, \bar{P},\bar{S})$ is contact scheme.
\end{theorem}
\begin{proof}
It follows from Theorem 3.1 that  the mapping $(q,p,s)\rightarrow (Q, P, S)$ given by $(\ref{3.3})$ is the phase flow of the stochastic contact Hamiltonian systems $(\ref{2.1})$.
Then this mapping is contact structure-preserving.
And by the numerical implementation of $(Q,P,S)$, the numerical scheme $(\bar{Q}, \bar{P},\bar{S})$ is the numerical approximations of $(Q,P,S)$, which can be  any order of accuracy. Therefore, the numerical scheme $(\bar{Q}, \bar{P},\bar{S})$ is contact scheme in the sense of discrete version. The proof is finished.
\end{proof}
\subsection{The construction of the contact Euler scheme of order 1.0}
We can obtain the recurrence formula to determining $G_\alpha$. Here we only consider the case $m=1$. To obtain the contact Euler scheme of order 1.0, we only need to truncate the series of $S$ by keeping  the needed terms, and we can have
 \begin{equation}\label{4.7}
 \begin{aligned}
\bar{S}&=G_{(1)}(\bar{Q})J_{(1)}(t)+G_{(1,1)}(\bar{Q})J_{(1,1)}(t)+G_{(0)}(\bar{Q})J_{(0)}(t)\\
       &=G_{(1)}(\bar{Q})W(t)+G_{(1,1)}(\bar{Q})\int_0^tW(t)\circ dW(t)+G_{(0)}(\bar{Q})t.
\end{aligned}
\end{equation}
Then we have
 \begin{equation}\label{4.8}
 \begin{aligned}
\frac{\partial\bar{S}}{\partial t}&= \frac{dC^{(0)}_0(t)}{dt}+\frac{dC^{(0)}_1(t)}{dt}Q+,...,+\frac{dC^{(0)}_k(t)}{dt}Q^k\\
&+\Big[G_{(1)}(\bar{Q})+G_{(1,1)}(\bar{Q})\cdot W(t)\Big]\cdot\dot{W}\\
\end{aligned}
\end{equation}
That is,
 \begin{equation}\label{4.9}
 \begin{aligned}
\bar{A}:=&\frac{dC^{(0)}_0(t)}{dt}+\frac{dC^{(0)}_1(t)}{dt}Q+,...,+\frac{dC^{(0)}_k(t)}{dt}Q^k=H_0\\
\bar{B}:=&G_{(1)}(\bar{Q})+G_{(1,1)}(\bar{Q})\cdot W(t)=H_1.\\
\end{aligned}
\end{equation}
Therefore, by compared with expressions of $H_0$ and $H_1$, we can get the numerical approximations of  $G_{(0)}t$, $G_{(1,1)}$ and $G_{(1)}$, respectively.

\section{Convergence Analysis}
In this section, we only focus on the convergence of the 1.0 order contact scheme by the generating function deduced by $(\ref{4.8})$. To simplify the notions, we here consider the stochastic contact Hamilton systems with $m=1$ and $n=1$, that is,
  \begin{equation}\label{5.1}
\left \{
\begin{aligned}
 dQ(t)=&\frac{\partial H_0}{\partial P}(Q(t),P(t),S,t)dt
  + \frac{\partial H_1}{\partial P}(Q(t),P(t),S,t)\circ dW(t),\\
 dP(t)=&-\Big[\frac{\partial H_0}{\partial Q}(Q(t),P(t),S,t)+P(t)\frac{\partial H_0}{\partial S}(Q(t),P(t),S,t)\Big ]dt\\
 &-\Big[\frac{\partial H_1}{\partial Q}(Q(t),P(t),S,t)+P(t)\frac{\partial H_1}{\partial S}(Q(t),P(t),S,t)\Big]\circ dW(t),\\
 dS=&\Big[P(t)\frac{\partial H_0}{\partial P}(Q(t),P(t),S,t)-H_0(Q(t),P(t),S,t) \Big ] dt\\
 &+\Big[P(t)\frac{\partial H_1}{\partial P}(Q(t),P(t),S,t)-H_1(Q(t),P(t),S,t)\Big]\circ dW(t).
 \end{aligned}
\right.
\end{equation}
 It is trivial that this method can be extended to the general case, that is,$m>1$ and $n>1$.

The following lemma is directly cited from \cite{Deng}.
\begin{lemma}
Assume that the Hamiltonian functions $H_0$, $H_1$ and their partial derivatives up to order one are continuous, and the following inequalities hold for some positive constants $K_i,i=1,2,3$,
\begin{equation}\label{5.2}
 \begin{aligned}
&\sum_{j=0}^1\bigg\lvert  H_j(\bar{Q},\bar{P},\bar{S},t)- H_j(Q(u),P(u),S,u)\bigg\rvert\\
&\leq K_1(\lvert\bar{Q}-Q(u) \rvert+\lvert\bar{P}-P(u) \rvert+\lvert\bar{S}-S\rvert),
 \end{aligned}
\end{equation}

\begin{equation}\label{5.3}
 \begin{aligned}
&\sum_{j=0}^1\bigg[\bigg\lvert \frac{\partial H_j}{\partial P}(\bar{Q},\bar{P},\bar{S},t)-\frac{\partial H_j}{\partial P}(Q(u),P(u),S,u)\bigg\rvert\\
&+\bigg \lvert \frac{\partial H_j}{\partial Q}(\bar{Q},\bar{P},\bar{S},t)-\frac{\partial H_j}{\partial Q}(Q(u),P(u),S,u)\bigg\rvert\\
&+\bigg \lvert \frac{\partial H_j}{\partial S}(\bar{Q},\bar{P},\bar{S},t)-\frac{\partial H_j}{\partial S}(Q(u),P(u),S,u)\bigg\rvert \bigg]\\
&\leq K_2(\lvert\bar{Q}-Q(u) \rvert+\lvert\bar{P}-P(u) \rvert+\lvert\bar{S}-S \rvert),
 \end{aligned}
\end{equation}

\begin{equation}\label{5.4}
 \begin{aligned}
&\sum_{j=0}^1\bigg[\bigg\lvert P(t)\frac{\partial H_j}{\partial P}(\bar{Q},\bar{P},\bar{S},t)-P(u)\frac{\partial H_j}{\partial P}(Q(u),P(u),S,u)\bigg\rvert\\
&+\bigg \lvert P(t) \frac{\partial H_j}{\partial S}(\bar{Q},\bar{P},\bar{S},t)-P(u)\frac{\partial H_j}{\partial S}(Q(u),P(u),S,u)\bigg\rvert\bigg]\\
&\leq K_3(\lvert\bar{Q}-Q(u) \rvert+\lvert\bar{P}-P(u) \rvert+\lvert\bar{S}-S \rvert ).
 \end{aligned}
\end{equation}
\end{lemma}

\begin{lemma}
Suppose that $h$ is the step size of the scheme, by the equations $(\ref{4.1})$ and  $(\ref{4.7})$, we have
\begin{equation}
 \begin{aligned}
 \lvert\bar{S}-S \rvert =O(h)\\
 \end{aligned}
 \nonumber
\end{equation}
Then we have
\begin{equation}
 \begin{aligned}
\mathbb{E}\big(\lvert\bar{Q}-Q(u) \rvert+\lvert\bar{P}-P(u) \rvert+\lvert\bar{S}-S \rvert \big)=O(h).
 \end{aligned}
 \nonumber
\end{equation}
\end{lemma}
\begin{proof}
It follows from the equations $(\ref{4.1})$ and  $(\ref{4.7})$ that the first conclusion holds. By the numerical approximation of $P$, we have
$$\lvert\bar{P}-P(u) \rvert =\lvert\frac{\partial \bar{S}}{\partial Q }-\frac{\partial S}{\partial Q }\rvert =O(h).$$ Utilizing Lemma 5.1 and the last two equations in  $(\ref{3.3})$, it is easy to obtain
$$\lvert\bar{Q}-Q(u) \rvert =O(h).$$
Therefore, this completes the proof.

\end{proof}

The following lemma is cited  from Theorem 1.1 in  \cite{Milstein01}.
\begin{lemma}
Let $\hat{\varphi}(u,t)$ be an approximate flow of SDE $(\ref{5.1})$ defined for values
$u \leq t$ such that $u, t \in \{h, 2h, ... , T\}$, where $h$ is the step size of the scheme. Suppose that the expectation of
the local error of the approximation is of order $\xi_1$, and that the local mean-square error is of order $\xi_2$, i.e. for
any $t < T$ there exists a constant $K$ such that
\begin{equation}
\lvert \mathbb{E}(\varphi_{t,t+h}(x)-\hat{\varphi}_{t,t+h}(x))\rvert \leq K(1+\lvert x \rvert^2)^{\frac{1}{2}}h^{\xi_1},\nonumber
\end{equation}
$$\|\varphi_{t,t+h}(x)-\hat{\varphi}_{t,t+h}(x)\| \leq K(1+\lvert x \rvert^2)^{\frac{1}{2}} h^{\xi_2},$$
where $\xi_2\geq \frac{1}{2}$ and $\xi_1\geq  \xi_2 + \frac{1}{2}$. Suppose further that $\hat{\varphi}_{u,t}$ is independent of $\mathcal{F}_u$ for all $u < t$. Then the approximation converges globally to the exact flow with strong order $\xi_2-\frac{1}{2}$ , i.e. there exists a constant
$K$ such that for all $t\in \{h, 2h, ... , T\}$,
 we have
 $$\|\varphi_t(x_0)-\hat{\varphi}_t(x_0)\|\leq K(1+\|x_0\|^2 )^\frac{1}{2}h^{\xi_2-\frac{1}{2}}.$$
\end{lemma}

Here we arrive at the convergence theorem of the scheme of $(\ref{4.7})$.
\begin{theorem}
If the contact Hamiltonian systems $(\ref{5.1})$  satisfy the conditions of Lemma 5.1, the numerical scheme of $(\ref{4.7})$ for the systems $(\ref{5.1})$ based on one-step approximation is of mean-square convergence order 1.0 by truncating the series $(\ref{4.7})$.
\end{theorem}

\begin{proof}
We define
$$R=\hat{\varphi}-\varphi.$$
From $(\ref{5.1})$, we have
 \begin{equation}\label{5.5}
\begin{split}
R=&\Bigg (
\begin{array}{cccc}
D_{11} \\
D_{21}\\
D_{31}\\
\end{array} \Bigg)
+\Bigg(
\begin{array}{cccc}
D_{12} \\
D_{22}\\
D_{32}
\end{array} \Bigg),
 \end{split}
\end{equation}
where
  \begin{equation}
\begin{aligned}
D_{11}=&\frac{\partial H_0}{\partial P}(\bar{Q},\bar{P},\bar{S},t)h-\int_{t}^{t+h} \frac{\partial H_0}{\partial P}(Q(u),P(u),S,u)du,\\
 D_{12}=&\frac{\partial H_1}{\partial P}(\bar{Q},\bar{P},\bar{S},t)\Delta W_t-\int_{t}^{t+h} \frac{\partial H_1}{\partial P}(Q(u),P(u),S,u)\circ dW(u),\\
 D_{21}=&\Big[\frac{\partial H_0}{\partial Q}(\bar{Q},\bar{P},\bar{S},t)+P(t)\frac{\partial H_0}{\partial S}(\bar{Q},\bar{P},\bar{S},t)\Big]h\\
      &-\Big[\int_{t}^{t+h} \frac{\partial H_0}{\partial Q}(Q(u),P(u),S,u)+P(u)\frac{\partial H_0}{\partial S}(Q(u),P(u),S,u)\Big]du,\\
D_{22}=&\Big[\frac{\partial H_1}{\partial Q}(\bar{Q},\bar{P},\bar{S},t)+P(t)\frac{\partial H_1}{\partial S}(\bar{Q},\bar{P},\bar{S},t)\Big]\Delta W_t\\
     &-\Big[\int_{t}^{t+h} \frac{\partial H_1}{\partial Q}(Q(u),P(u),S,u)+P(u)\frac{\partial H_1}{\partial S}(Q(u),P(u),S,u)\Big]\circ dW(u),\\
 D_{31}=&\Big[P(t)\frac{\partial H_0}{\partial P}(\bar{Q},\bar{P},\bar{S},t)-H_0(\bar{Q},\bar{P},\bar{S},t) \Big ]h\\
& +\int_{t}^{t+h} \Big[P(u)\frac{\partial H_0}{\partial P}(Q(u),P(u),S,u)-H_0(Q(u),P(u),S,u) \Big ]du,\\
  \end{aligned}
 \nonumber
\end{equation}
and
\begin{equation}
\begin{aligned}
D_{32}=&\Big[P(t)\frac{\partial H_1}{\partial P}(\bar{Q},\bar{P},\bar{S},t)-H_1(\bar{Q},\bar{P},\bar{S},t) \Big ]\Delta W_t\\
& +\int_{t}^{t+h} \Big[P(u)\frac{\partial H_1}{\partial P}(Q(u),P(u),S,u)-H_1(Q(u),P(u),S,u) \Big ]\circ dW(u).
 \end{aligned}
 \nonumber
\end{equation}

We will show that the  numerical scheme $(\ref{4.7})$ possesses the mean-square convergence order 1.0 according to Lemma 5.3(Generalized Milstein Theorem) by estimating $\lvert \mathbb{E} R\lvert$ and $\|R\|:=[\mathbb{E} \lvert R\lvert^2]^{\frac{1}{2}}$.

First,
\begin{equation}\label{5.6}
\begin{split}
\lvert \mathbb{E} R\lvert \leq &\mathbb{E} \Bigg \lvert
\begin{array}{c}
D_{11} \\
D_{21}\\
D_{31}\\
\end{array} \Bigg\lvert
+\mathbb{E} \Bigg\lvert
\begin{array}{c}
D_{12} \\
D_{22}\\
D_{32}
\end{array} \Bigg\lvert.
 \end{split}
\end{equation}

Consider
\begin{equation}
\mathbb{E} \Big \lvert D_{11} \Big \lvert \leq I_{11},\nonumber
\end{equation}
where
$$I_{11}= \mathbb{E} \int_{t}^{t+h} \Big \lvert\frac{\partial H_0}{\partial P}(\bar{Q},\bar{P},\bar{S},t)-\frac{\partial H_0}{\partial P}(Q(u),P(u),S,u)\Big \lvert du.$$
%and
%$$I_2=\mathbb{E} \int_{t}^{t+h}\Big \lvert \sigma_0(\bar{P},Q(s))-\sigma_0(P(s),Q(s))\Big \lvert ds .$$
Then it follows from Lemma 5.1 and Lemma 5.2 that
\begin{equation}\label{5.7}
\begin{split}
I_{11}&\leq K_2 \mathbb{E}\int_{t}^{t+h}\Big [\lvert\bar{Q}-Q(u) \rvert+\lvert\bar{P}-P(u) \rvert+\lvert\bar{S}-S \rvert\Big ] du \\
 &=  K_2 \int_{t}^{t+h}O(h) du =O(h^2).
\end{split}
\end{equation}
Similarly, we have
\begin{equation}
\mathbb{E} \Big \lvert D_{12} \Big \lvert \leq I_{12},\nonumber
\end{equation}
where
$$I_{12}= \mathbb{E} \int_{t}^{t+h} \Big \lvert\frac{\partial H_1}{\partial P}(\bar{Q},\bar{P},\bar{S},t)-\frac{\partial H_1}{\partial P}(Q(u),P(u),S,u)\Big \lvert \circ dW(u).$$
Then by Lemma 5.1 and Lemma 5.2 we have
\begin{equation}\label{5.8}
\begin{split}
I_{12}&\leq K_2 \mathbb{E}\int_{t}^{t+h}\Big [\lvert\bar{Q}-Q(u) \rvert+\lvert\bar{P}-P(u) \rvert+\lvert\bar{S}-S \rvert\Big]\circ dW(u) \\
 &=K_2 \int_{t}^{t+h}O(h) du =O(h^2).
\end{split}
\end{equation}
By the same way, we can get some similar results as follows,
\begin{equation}\label{5.9}
\mathbb{E} \Big \lvert D_{21} \Big \lvert= O(h^2),\mathbb{E} \Big \lvert D_{22} \Big \lvert= O(h^2),\mathbb{E} \Big \lvert D_{31} \Big \lvert= O(h^2),
\end{equation}
and
$$\mathbb{E} \Big \lvert D_{32} \Big \lvert= O(h^2).$$
Therefore, combining $(\ref{5.7})$-$(\ref{5.9})$ to $(\ref{5.6})$,  we have we can have
\begin{equation}\label{5.10}
\lvert\mathbb{E}   R  \lvert= O(h^2).
\end{equation}

Second, according to the definition of the norm,  we have
\begin{equation}\label{5.11}
\begin{split}
\|R\|^2:=\mathbb{E} \lvert R\lvert^2\leq
2\mathbb{E} \Bigg\lvert
\begin{array}{c}
D_{11} \\
D_{21}\\
D_{31}\\
\end{array} \Bigg\lvert^2
+2\mathbb{E} \Bigg\lvert
\begin{array}{c}
D_{12} \\
D_{22}\\
D_{32}
\end{array} \Bigg\lvert^2.
 \end{split}
\end{equation}

We also consider
\begin{equation}\label{5.12}
\mathbb{E} \Big \lvert D_{11} \Big \lvert^2 \leq I^*_{11},\nonumber
\end{equation}
where
$$I^*_{11}=\mathbb{E}\Big[\int_{t}^{t+h} \Big \lvert\frac{\partial H_0}{\partial P}(\bar{Q},\bar{P},\bar{S},t)-\frac{\partial H_0}{\partial P}(Q(u),P(u),S,u)\Big \lvert du\Big]^2.$$

Then it follows from Lemma 5.1, Lemma 5.2  and Cauchy-Schwarz inequality that
\begin{equation}\label{5.13}
\begin{split}
I^*_{11}&\leq K_1 \mathbb{E}\int_{t}^{t+h}\Big [\lvert\bar{Q}-Q(u) \rvert+\lvert\bar{P}-P(u) \rvert+\lvert\bar{S}-S \rvert\Big ] ^2du \\
 &=K_1 \mathbb{E}\int_{t}^{t+h}O(h^2) du =O(h^3).
\end{split}
\end{equation}
Similarly, we have
\begin{equation}
\mathbb{E} \Big \lvert D_{12} \Big \lvert^2 \leq I^*_{12},\nonumber
\end{equation}
where
$$I^*_{12}=\mathbb{E} \Big[\int_{t}^{t+h} \Big \lvert\frac{\partial H_1}{\partial P}(\bar{Q},\bar{P},\bar{S},t)-\frac{\partial H_1}{\partial P}(Q(u),P(u),S,u)\Big \lvert \circ dW(u)\Big]^2.$$
By Cauchy-Schwarz inequality, Lemma 5.1, Lemma 5.2, the relation of It\^{o} and Stratonovich integrals and It\^{o} isometry theorem, then we have
\begin{equation}\label{5.14}
\begin{split}
I^*_{12}\leq &2\mathbb{E}\int_{t}^{t+h} \Big \lvert\frac{\partial H_1}{\partial P}(\bar{Q},\bar{P},\bar{S},t)-\frac{\partial H_1}{\partial P}(Q(u),P(u),S,u)\Big \lvert^2du\\
       &+\frac{1}{2}h \mathbb{E}\int_{t}^{t+h}\bigg[\frac{\partial }{\partial t}\Big \lvert\frac{\partial H_1}{\partial P}(\bar{Q},\bar{P},\bar{S},t)-\frac{\partial H_1}{\partial P}(Q(u),P(u),S,u)\Big\lvert\\
       &\cdot \Big \lvert\frac{\partial H_1}{\partial P}(\bar{Q},\bar{P},\bar{S},t)-\frac{\partial H_1}{\partial P}(Q(u),P(u),S,u)\Big \lvert\Bigg]^2 du \\
        \leq& 2 K_2 \mathbb{E}\int_{t}^{t+h}\Big [\lvert\bar{Q}-Q(u) \rvert+\lvert\bar{P}-P(u) \rvert+\lvert\bar{S}-S \rvert\Big]^2du\\
         &+\frac{1}{2}K_2h\mathbb{E}\int_{t}^{t+h}\bigg[\frac{\partial }{\partial t}\Big \lvert\frac{\partial H_1}{\partial P}(\bar{Q},\bar{P},\bar{S},t)-\frac{\partial H_1}{\partial P}(Q(u),P(u),S,u)\Big\lvert^2\\
         & \cdot \Big [\lvert\bar{Q}-Q(u) \rvert+\lvert\bar{P}-P(u) \rvert+\lvert\bar{S}-S \rvert\Big]^2 du \\
 \leq& 2 K_2 \mathbb{E}\int_{t}^{t+h}O(h^2)du+  \frac{1}{2}K_2 h\mathbb{E}\int_{t}^{t+h}O(h^2) du =O(h^3).
\end{split}
\end{equation}
By the same way, we can get some similar results as follows,
\begin{equation}\label{5.15}
\mathbb{E} \Big \lvert D_{21} \Big \lvert^2= O(h^2),\mathbb{E} \Big \lvert D_{22} \Big \lvert^2= O(h^2),\mathbb{E} \Big \lvert D_{31} \Big \lvert^2= O(h^3),
\end{equation}
and
$$\mathbb{E} \Big \lvert D_{32} \Big \lvert^2= O(h^3).$$
Therefore, combining $(\ref{5.12})$-$(\ref{5.15})$ to $(\ref{5.11})$, we can have
\begin{equation}
\mathbb{E} \lvert  R  \lvert^2= O(h^3).\nonumber
\end{equation}
That is,
\begin{equation}\label{5.16}
\|R\|=\Big[\mathbb{E}\lvert  R\lvert^2\Big]^{\frac{1}{2}} =O(h^{\frac{3}{2}}).
\end{equation}

Thus, by $(\ref{5.10})$ and $(\ref{5.16})$, we have $\xi_1=2$ and $\xi_2=\frac{3}{2}$. According to Lemma 5.3, the numerical scheme $(\ref{4.7})$  is of mean-square convergence order $1.0$.

This completes the proof of Theorem 5.4.

\end{proof}

\section{Numerical Experiments}

We apply Theorem 3.1 and the method in Section 4 to
construct contact structure-preserving schemes via stochastic contact Hamilton-Jacobi equation for stochastic contact Hamiltonian systems. The schemes present the validity of stochastic contact Hamilton-Jacobi equation and the effectiveness of the numerical experiments.

We consider the following stochastic damped parametric oscillator, which has mass m and time-dependent frequency $\omega(t)$, and whose variables are one-dimensional.
The related stochastic contact Hamilton systems is in the form of,
  \begin{equation}\label{6.0}
\left \{
\begin{aligned}
 dQ&=\frac{P}{m}dt\\
 dP&=-\Big[m\omega^2(t)Q+\gamma P\Big ]dt-\psi'(Q)\circ dW(t),\\
 dS&=\Big[\frac{1}{2m}P^2-\frac{1}{2}m\omega^2(t)Q^2-\gamma S \Big ] dt-\psi(Q)\circ dW(t),
 \end{aligned}
\right.
\end{equation}
where the variables $Q,P,S\in \mathbb{R}$ , $m,\omega$ and $\gamma$ are parameters, $\psi(Q):=a Q$, and
$$H_0=\frac{1}{2m}P^2+\frac{1}{2}m\omega^2(t)Q^2+\gamma S, H_1=\psi(Q).$$
Its contact Hamiltonian is
\begin{equation}\label{6.1}
\begin{aligned}
\mathcal{H}=\frac{1}{2m}P^2+\frac{1}{2}m\omega^2(t)Q^2+\gamma S.
 \end{aligned}
\nonumber
 \end{equation}
And the stochastic contact Hamilton-Jacobi equation  driven by  Gaussian noise is
\begin{equation}\label{6.2}
\begin{aligned}
dS+\Big[\frac{1}{2m}P^2+\frac{1}{2}m\omega^2(t)Q^2+\gamma S\Big]dt+\psi(Q)\circ dW_t =0.
 \end{aligned}
%\nonumber
 \end{equation}

\subsection{Construction of contact structure-preserving  schemes}

It is easy to check that $H_0$, $H_1$ and their partial derivatives satisfy Lemma 5.1. Therefore, utilizing the method in Section 4, we choose the ansatz of $S(Q,C,t)$ with the order $1.0$,  then we have
 \begin{equation}\label{6.3}
\begin{aligned}
\bar{S}(\bar{Q},C,t)=&G_{(1)}(\bar{Q})W(t)+G_{(1,1)}(\bar{Q})\int_0^tW(t)\circ dW(t)+G_{(0)}(\bar{Q})t,\\
\end{aligned}
\end{equation}
where
$$G_{(0)}(\bar{Q})t=C^{(0)}_0(t)+C^{(0)}_1(t)\bar{Q}+C_2^{(0)}(t)\bar{Q}^2.$$

Differentiating two sides of $\bar{S}(\bar{Q},C,t)$ with respect to $t$, then we obtain
$$\frac{\partial \bar{S}}{\partial t}=A_1+B_1\dot{W},$$
 where
\begin{equation}
  \left \{
\begin{aligned}
A_1=&\frac{d C_0^{(0)}(t)}{dt}+\frac{d C_1^{(0)}(t)}{dt}\bar{Q}+\frac{d C_2^{(0)}(t)}{dt}\bar{Q}^2 ,\\
B_1=&G_{(1)}(\bar{Q})+G_{(1,1)}(\bar{Q})\cdot W(t).\nonumber
\end{aligned}
\right.
\end{equation}

Comparing $(\ref{6.3})$  with $(\ref{6.1})$ about the corresponding same terms, we get
 \begin{equation}
 \left \{
\begin{aligned}
A_1=-H_0\\
B_1=-H_1.\nonumber
\end{aligned}
\right.
\end{equation}
That is,
 \begin{equation}
  \left \{
\begin{aligned}
-H_0&=\frac{d C_0^{(0)}(t)}{dt}+\frac{d C_1^{(0)}(t)}{dt}\bar{Q}+\frac{d C_2^{(0)}(t)}{dt}\bar{Q}^2 \\
&=-\frac{1}{2m}\bar{P}^2-\frac{1}{2}m\omega^2(t)\bar{Q}^2-\gamma \bar{S},\\
-H_1&=G_{(1)}(\bar{Q})+G_{(1,1)}(\bar{Q})\cdot W(t)\\
&=-\psi(\bar{Q}).
\end{aligned}
\right.
\nonumber
\end{equation}

By comparing the expressions of the above equations, we have
 \begin{equation}\label{6.4}
  \left \{
\begin{aligned}
&\frac{d C_0^{(0)}(t)}{dt}+\frac{d C_1^{(0)}(t)}{dt}\bar{Q}+\frac{d C_2^{(0)}(t)}{dt}\bar{Q}^2 =-\frac{1}{2m}\bar{P}^2-\frac{1}{2}m\omega^2(t)\bar{Q}^2-\gamma \bar{S},\\
&G_{(1)}(\bar{Q})=-\psi(\bar{Q})=-\bar{Q},\\
&G_{(1,1)}(\bar{Q})=0.
\end{aligned}
\right.
\end{equation}
According to $(\ref{6.3})$ and $(\ref{6.4})$, we  obtain
 \begin{equation}\label{6.5}
\begin{aligned}
\bar{S}(\bar{Q},C,t)=&-\psi(\bar{Q})W(t)+G_{(0)}(\bar{Q})t\\
                    =& -a\bar{Q}W(t)+C^{(0)}_0(t)+C^{(0)}_1(t)\bar{Q}+C_2^{(0)}(t) \bar{Q}^2.
\end{aligned}
\end{equation}
It follows from $(\ref{6.5})$ and the last two equations of $(\ref{6.4})$  we have
 \begin{equation}\label{6.6}
\begin{aligned}
\bar{P}(t)=&\frac{\partial \bar{S}(\bar{Q},C,t)}{\partial \bar{Q}}\\
=&-aW(t)+C^{(0)}_1(t)+2C_2^{(0)}(t)\bar{Q}.\\
\end{aligned}
\end{equation}
Then $(\ref{6.5})$ and $(\ref{6.6})$ are inserted into the first equation of $(\ref{6.4})$, after the  order in $\bar{Q}$ is compared, we obtain the following conditions,
 \begin{equation}\label{6.7}
  \left \{
\begin{aligned}
\frac{d C_2^{(0)}(t)}{dt}=&-\frac{2}{m}\Big[C_2^{(0)}(t)\Big]^2-\gamma C_2^{(0)}(t)-\frac{1}{2}mw^2(t)\\
\frac{d C_1^{(0)}(t)}{dt}=&\Big[ -\frac{2}{m}C_2^{(0)}(t)-\gamma \Big]\cdot \Big[C_1^{(0)}(t)-a W(t)\Big]\\
\frac{d C_0^{(0)}(t)}{dt}=&-\frac{1}{2m}\Big[(C_1^{(0)}(t))^2+(aW(t))^2-2aW(t)C_1^{(0)}(t)\Big]\\
                          &-\gamma C_0^{(0)}(t),\\
\end{aligned}
\right.
\end{equation}
According to the third equation in $(\ref{3.3})$, we have $$b(t)=b_0\exp(-\gamma t).$$
Based on $(\ref{6.5})$ and the second equation in $(\ref{3.3})$, we have
$$b(t)=\frac{\partial \bar{S}}{\partial C}=\frac{\partial C_0^{(0)}(t)}{\partial C}+\frac{\partial C_1^{(0)}(t)}{\partial C}\bar{Q}+\frac{\partial C_2^{(0)}(t)}{\partial C}\bar{Q}^2=b_0\exp(-\gamma t).$$

Therefore, we can obtain the numerical approximation of $\bar{Q}(t)$
 \begin{equation}\label{6.8}
\begin{aligned}
\bar{Q}(t)= \frac{-\chi+\sqrt{\chi^2-4\kappa \eta}}{2\kappa},\\
\end{aligned}
\end{equation}
where $$\kappa=\frac{\partial C_2^{(0)}(t)}{\partial C}, \chi=\frac{\partial C_1^{(0)}(t)}{\partial C},\eta=\frac{\partial C_0^{(0)}(t)}{\partial C}-b_0\exp(-\gamma t),$$
and $$b_0=b(0)=1+\bar{Q}_0+\bar{Q}_0^2.$$
Then applying  the numerical methods to  $(\ref{6.7})$, we obtain the numerical approximation of  $C_0^{(0)}(t)$, $C_1^{(0)}(t)$ and  $ C_2^{(0)}(t)$. So that we can get the numerical solutions $(\bar{Q}(t), \bar{P}(t),\bar{S})$ of $(\ref{6.0})$ by  $(\ref{6.8})$,  $(\ref{6.6})$ and $(\ref{6.5})$, respectively.

\subsection{Numerical implementations of contact structure-preserving schemes}  %,

Here we choose the initial conditions  as follows, $m=1,\gamma=1,\psi(Q)=Q$. We only consider  the damped parametric oscillator in the case of free particle with $\omega(t)=0$.
Then  $(\ref{6.0})$ reads
  \begin{equation}\label{6.9}
\left \{
\begin{aligned}
 dQ&=P dt,\\
 dP&=-P dt- dW(t),\\
 dS&=(\frac{1}{2}P^2- S)dt-Q\circ dW(t).
 \end{aligned}
\right.
\end{equation}
 We replace the time step t, $(q,p,s)$, and $(Q(t),P(t),S(t))$ by $h$, $(q_n,p_n,s_n)$ and $(q_{n+1},p_{n+1},s_{n+1})$, respectively. Then $(C, C_0^{(0)}(t))$, $(C, C_1^{(0)}(t))$ and  $(C, C_2^{(0)}(t))$ are rewritten by $(x_n,x_{n+1})$, $(y_n,y_{n+1})$ and $(z_n,z_{n+1})$, respectively. From $(\ref{6.7})$ we obtain
 \begin{equation}\label{6.10}
  \left \{
\begin{aligned}
z_{n+1}=&z_n-2z_{n}^2h-z_{n}h,\\
y_{n+1}=&y_n-2z_{n}(y_{n}-\Delta W_n)h-(y_{n}-\Delta W_n)h, \\
x_{n+1}=&x_n-\frac{1}{2}h\Big[(y_{n})^2+( \Delta W_n)^2-2 \Delta W_ny_{n}\Big]-x_{n}h.
\end{aligned}
\right.
\end{equation}
Then we have
 \begin{equation}\label{6.11}
  \left \{
\begin{aligned}
\frac{\partial z_{n+1}}{\partial z_n}=&1-h-4hz_n,\\
\frac{\partial y_{n+1}}{\partial y_n}=&1-h-2hz_{n}, \\
\frac{\partial x_{n+1}}{\partial x_n}=&1-h.
\end{aligned}
\right.
\end{equation}

Therefore, using $(\ref{6.11})$, $(\ref{6.8})$,  $(\ref{6.6})$ and $(\ref{6.5})$ we get the contact structure-preserving schemes for the equations $(\ref{6.9})$, which are in the form of
\begin{equation}\label{6.12}
\left \{
\begin{aligned}
 q_{n+1}=&\frac{-(1-h-2hz_n)+\sqrt{12h(1-h)z_n+4h^2z_n^2-3(1-h)^2}}{2(1-h-4hz_n)},\\
 p_{n+1}=&y_{n+1}+2z_{n+1}q_{n+1}-\Delta W_n,\\
 s_{n+1}=&x_{n+1}+y_{n+1}q_{n+1}+z_{n+1}q_{n+1}^2-q_{n+1}\Delta W_n.\\
\end{aligned}
\right.
\end{equation}
The Euler-Maruyama scheme for $(\ref{6.9})$ is
   \begin{equation}\label{6.13}
\left \{
\begin{aligned}
 q_{n+1}&=q_n+p_nh,\\
 p_{n+1}&=p_n-p_n h-\Delta W_n,\\
 s_{n+1}&=s_n+ \Big[\frac{1}{2}p_n^2-s_n\Big] h-\frac{1}{2}q_{n}h-q_n\Delta W_n.
 \end{aligned}
\right.
\end{equation}
The one-step contact structure-preserving scheme for $(\ref{6.9})$ by the method in \cite{Zhan4}, i.e., stochastic Herglotz variational principle, is in the form of
  \begin{equation}\label{6.14}
\left \{
\begin{aligned}
 q_{n+1}&=q_{n}+(1-\frac{1}{2}h)hp_{n}-h\Delta W_{n},\\
 p_{n+1}&=p_{n}-\frac{1}{1-\frac{h}{2}}\Delta W_{n},\\
 s_{n+1}&= s_n+\frac{1}{2h}(q_{n+1}-q_n)^2-\frac{1}{2}h(s_{n+1}+s_n)-q_n\Delta W_n.
 \end{aligned}
\right.
\end{equation}
where $h=\frac{T}{N},$ $\Delta W_n=W(t_{n+1})-W(t_n)$.

\subsection{Numerical results}
Dynamical behaviors of the solutions obtained by the schemes $(\ref{6.12})$, $(\ref{6.13})$  and $(\ref{6.14})$ can be compared in this numerical experiment.
For the sake of improving the comparability of the results, we choose the initial conditions as follows,  $T=20.0$, $h=0.1$,  $\varepsilon=0.02$, $\alpha=0.1$, $N=200.0$, $(q(0),p(0),s(0))=(0.75,-0.25,0.08)$ and $(x(0),y(0),z(0))=(0.65,0.65,0.65)$.

 {\bf 6.3.1 Comparison of sample trajectories with non-contact scheme $(\ref{6.13})$ }

  \begin{figure}[H]
   \centering
    \begin{minipage}{6.5cm}
       \includegraphics[width=2.8in, height=1.8in]{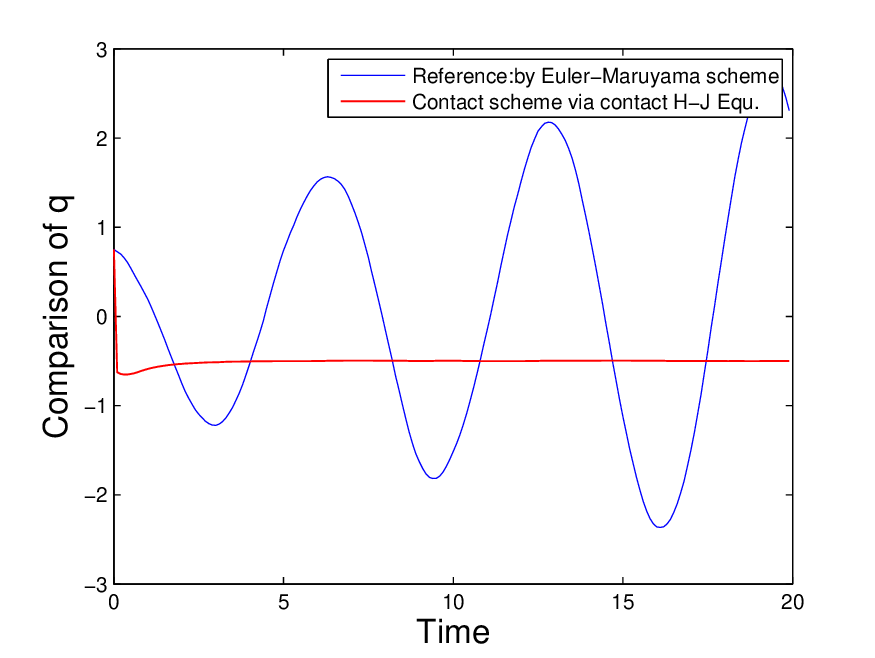}
    \end{minipage}
   \\
   \begin{minipage}{6.5cm}
       \includegraphics[width=2.8in, height=1.8in]{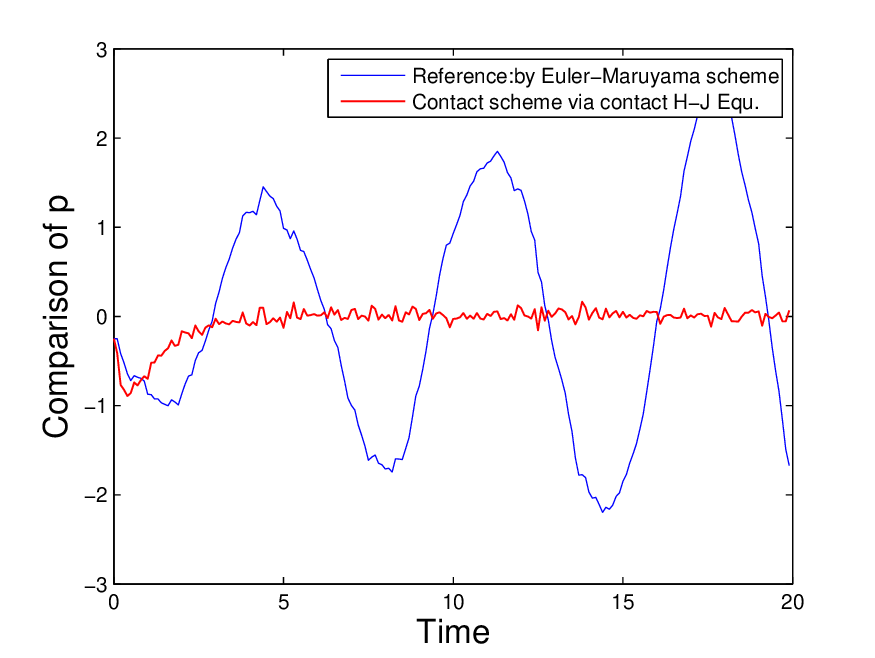}
    \end{minipage}
    \begin{minipage}{6.5cm}
       \includegraphics[width=2.8in, height=1.8in]{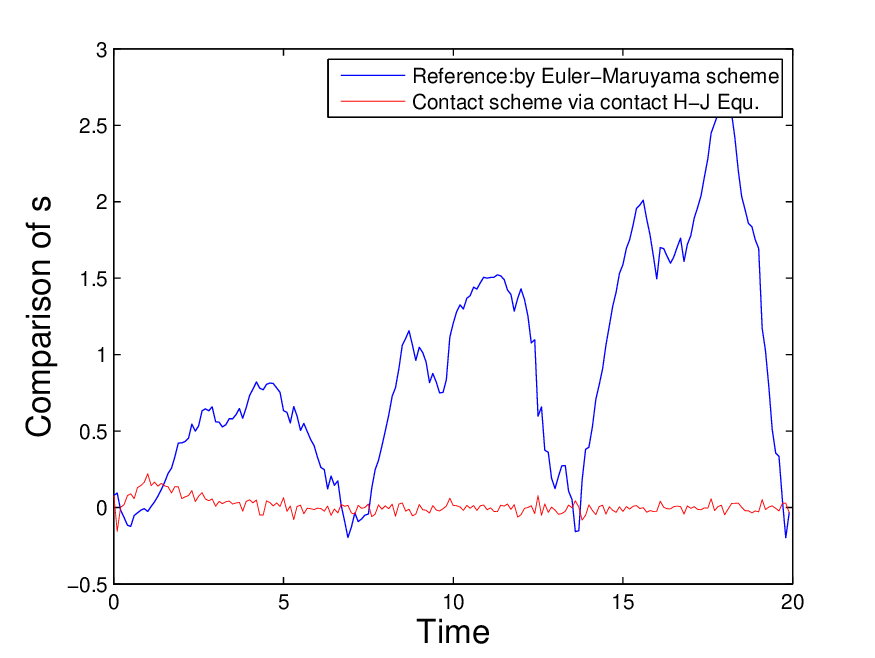}
    \end{minipage}
    \par{\scriptsize  Fig.1. Comparison between the trajectories of the equation $(\ref{6.9})$ by the schemes $(\ref{6.12})$  and $(\ref{6.13})$ in the coordinates $q,p$ and $s$, respectively.}
\end{figure}

 In order to compare our proposed scheme $(\ref{6.12})$ with the non-contact scheme $(\ref{6.13})$, we can compare long time behaviors of the numerical solutions obtained by the schemes $(\ref{6.12})$ and $(\ref{6.13})$. As we can see from Fig.1, the former shows better performance than the latter. That is to say, the non-contact scheme $(\ref{6.13})$ leads to   larger oscillations than the contact scheme $(\ref{6.12})$.

  {\bf 6.3.2 Comparison of sample trajectories with the contact scheme $(\ref{6.13})$}

   Here we aim to compare the scheme $(\ref{6.12})$ with the contact scheme $(\ref{6.14})$,which is constructed by Herglogz variational principle. It shows from Fig.2, they have similar behaviors, and the difference in the  trajectories in Fig.2 is less than this in Fig.1. That is to say, the scheme $(\ref{6.12})$  is also contact scheme.
   \begin{figure}[H]
   \centering
    \begin{minipage}{6.5cm}
       \includegraphics[width=2.8in, height=1.8in]{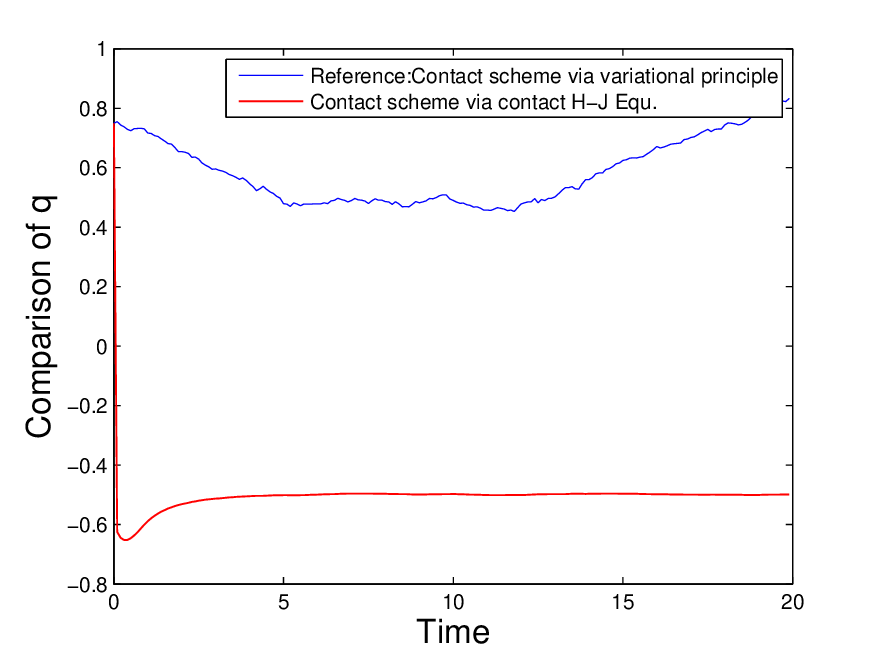}
    \end{minipage}
   \\
   \begin{minipage}{6.5cm}
       \includegraphics[width=2.8in, height=1.8in]{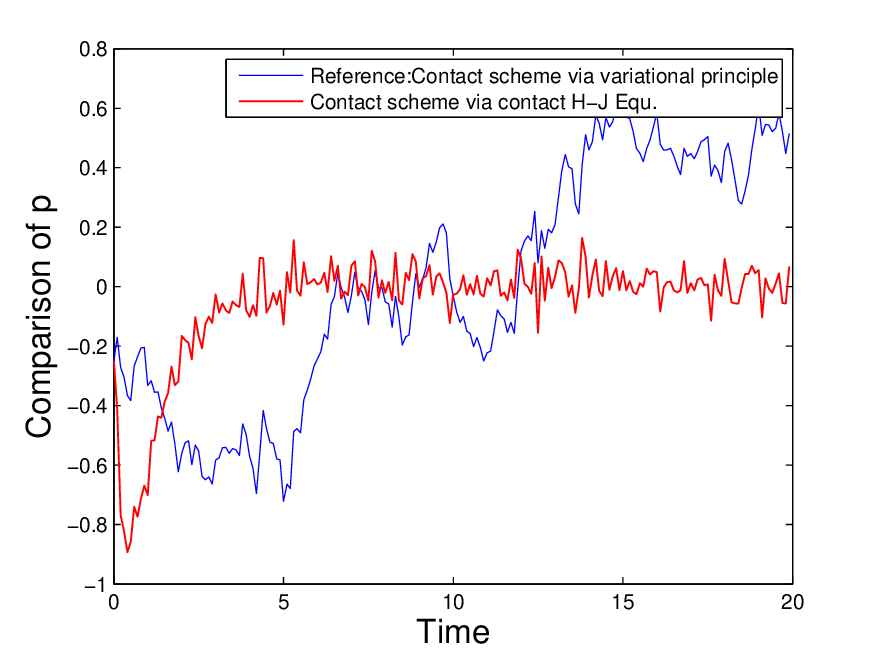}
    \end{minipage}
    \begin{minipage}{6.5cm}
       \includegraphics[width=2.8in, height=1.8in]{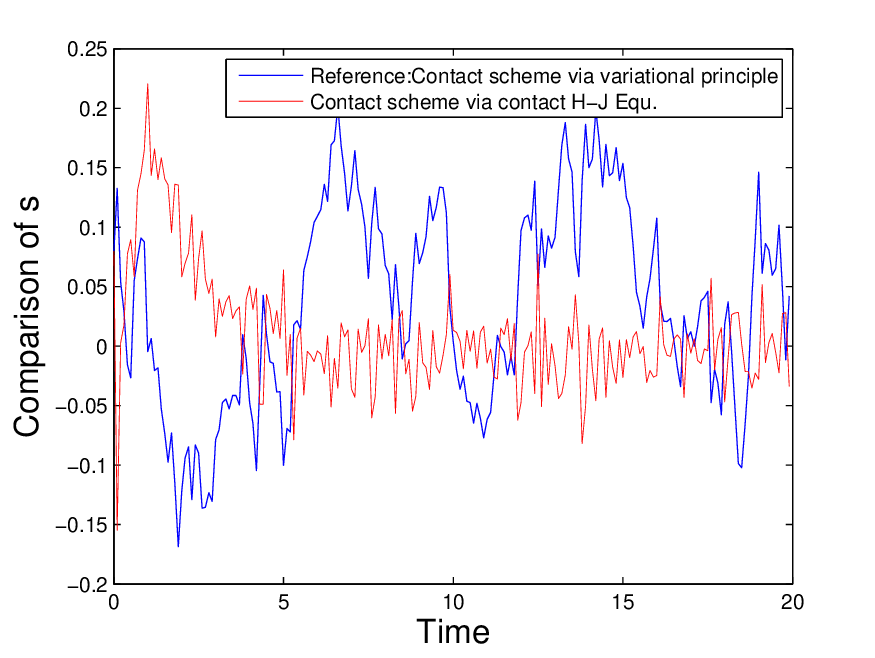}
    \end{minipage}
    \par{\scriptsize  Fig.2. Comparison between the trajectories of the equation $(\ref{6.9})$ by the schemes $(\ref{6.12})$  and $(\ref{6.14})$ in the coordinates $q,p$ and $s$, respectively.}
\end{figure}

{\bf 6.3.3 Preservation of contact structure of SDE $(\ref{6.9})$ }

    \begin{figure}[H]
   \centering
   \begin{minipage}{6.5cm}
       \includegraphics[width=2.8in, height=1.8in]{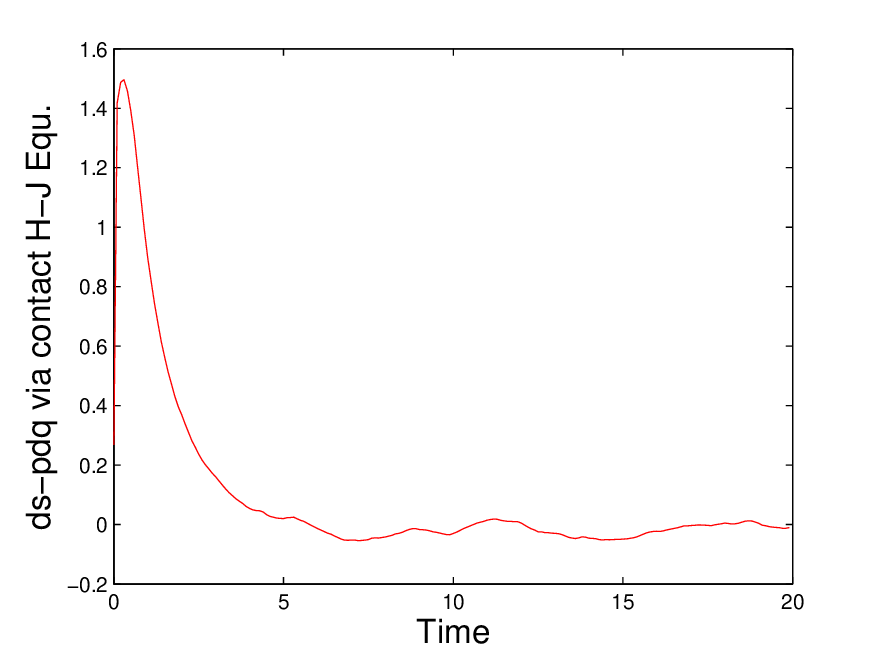}
    \end{minipage}
    \begin{minipage}{6.5cm}
       \includegraphics[width=2.8in, height=1.8in]{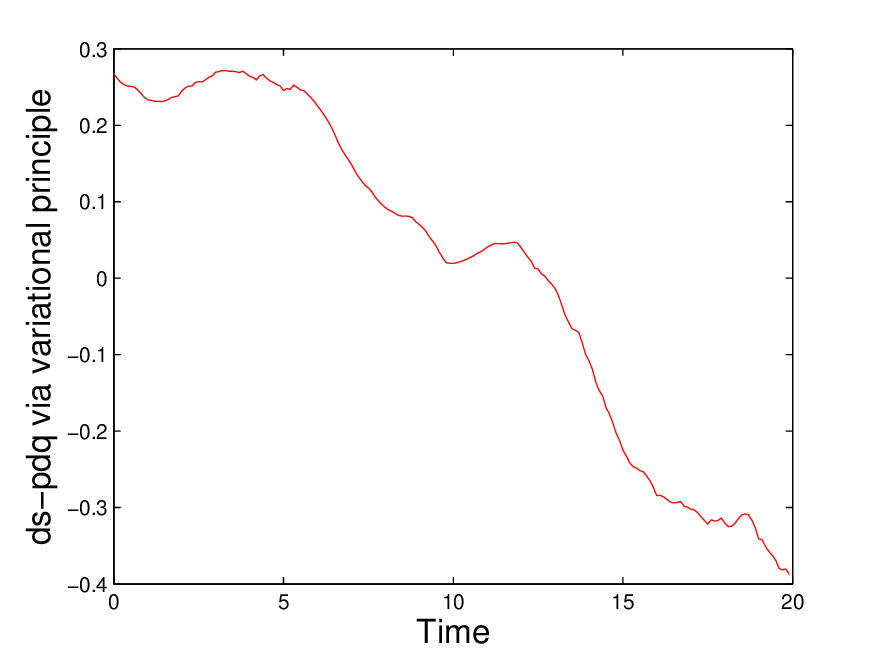}
    \end{minipage}
    \par{\scriptsize  Fig.3. Comparison the preservation of contact structure between the trajectories of the equation $(\ref{6.9})$ by the schemes $(\ref{6.12})$  and $(\ref{6.14})$, respectively.}
\end{figure}

    \begin{figure}[H]
   \centering
   \begin{minipage}{6.5cm}
       \includegraphics[width=2.8in, height=1.8in]{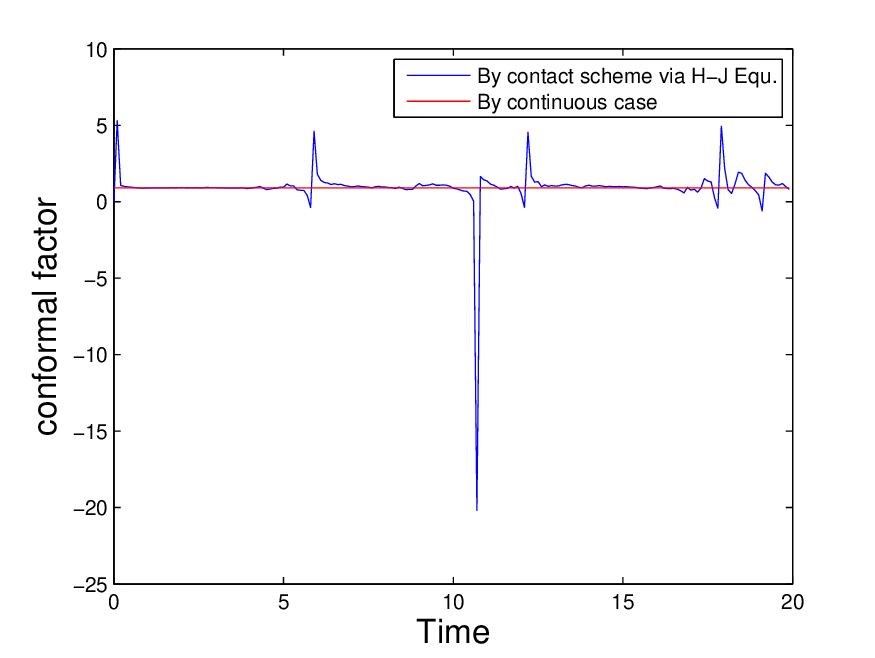}
    \end{minipage}
    \begin{minipage}{6.5cm}
       \includegraphics[width=2.8in, height=1.8in]{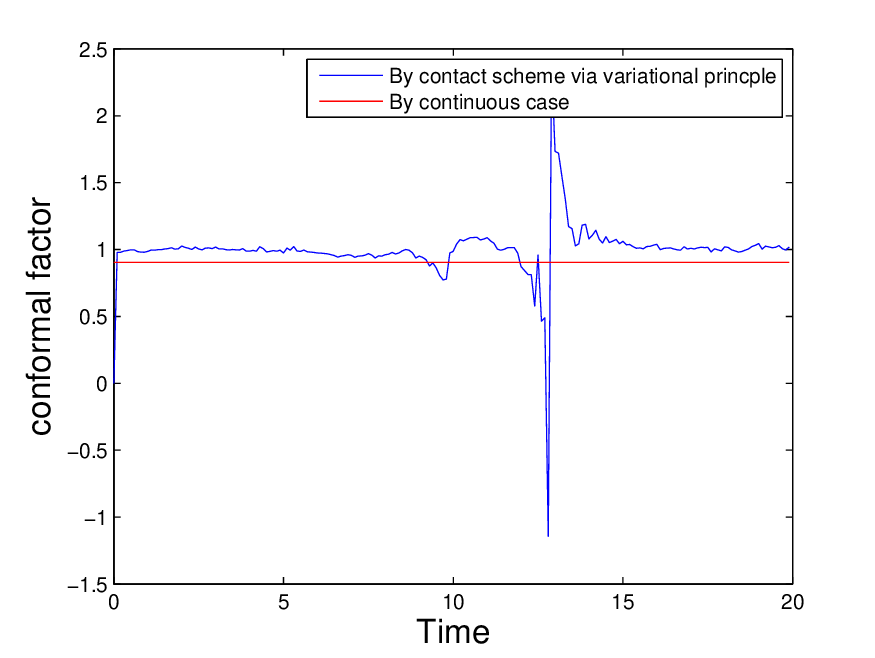}
    \end{minipage}
    \par{\scriptsize  Fig.4. Comparison of the conformal factor of  the continuous case and the contact schemes $(\ref{6.12})$  and $(\ref{6.14})$, respectively.}
\end{figure}

Fig.3 shows that the contact structure of $(\ref{6.9})$ is almost surely preserved in the time interval $[0,20]$ by the scheme $(\ref{6.12})$. Fig.3 also shows that  the trajectory of contact structure in the time interval $[0,20]$ in  the left halt-plot of Fig.3 has better performance than those in the right-plot. That is to say, it is flat almost surely in the time interval $[0,20]$. Fig.4 also shows the fact that the conformal factor is preserved well by the schemes $(\ref{6.12})$  and $(\ref{6.14})$.

 Therefore, these phenomena verify the results of Theorem 3.1. It indicates the fact that the  scheme $(\ref{6.12})$ is efficient.

  {\bf 6.3.4 Convergence of the contact scheme $(\ref{6.12})$}

  Fig. 5 shows the value $\|(Q(T)+P(T)+S(T))^2-(Q_N+P_N+S_N)^2\|$ against the time stepsize $h=0.02,0.04,0.06,0.08$ and $0.10$ with $T=120.0$, where the initial values are similar to Section 6.3.1. It also shows that the mean-square order of the scheme $(\ref{6.12})$ is 1.0, as indicated by the
reference line of slope 1.0.

\begin{figure}[H]
\centering
\includegraphics[width=3.5in, height=1.8in]{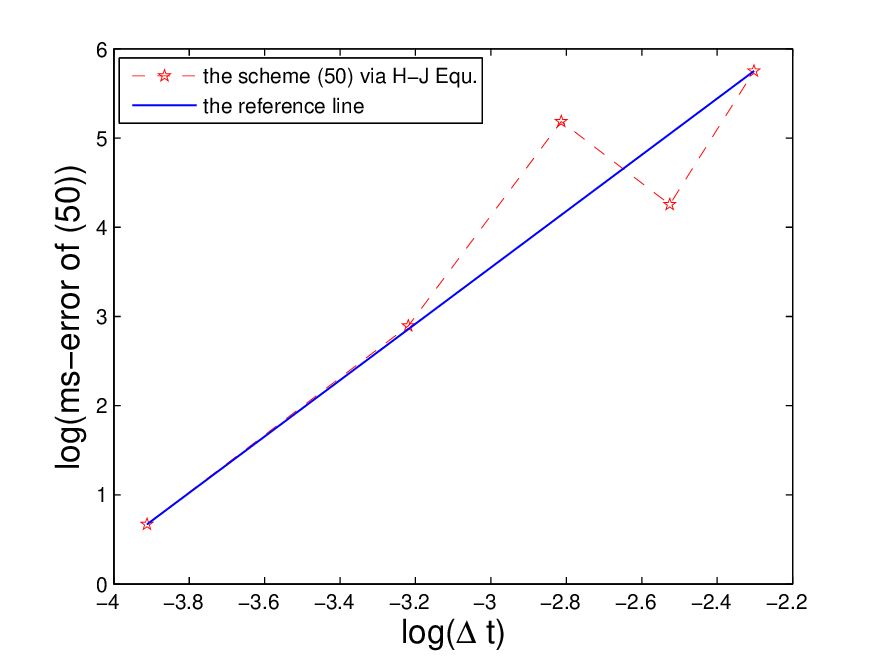}
\par{\scriptsize  Fig.5. The mean-square convergence order of the scheme $(\ref{6.12})$  and the time stepsize $h=0.02,0.04,0.06,0.08,0.10$.}
\end{figure}

\begin{table}[h]
 \centering
\begin{tabular}{ccccccccccccc}
\hline
$h$      &log(ms-error of (50))                      \\
\hline
0.02            &0.6684                          \\

0.04            &2.8935                                  \\

0.06            &5.1845                             \\

0.08            &4.2541                                         \\

0.10            &5.7526                                  \\
\hline
\end{tabular}
\label{objectiveResults}
\caption{The log (mean-square error of the scheme $(\ref{6.12})$ ) with the time stepsize $h=0.02,0.04,0.06,0.08,0.10$.}
\end{table}

\section{Conclusion}
 This paper focuses on the numerical dynamics of the stochastic contact Hamiltonian systems via structure-preserving methods. The contact structure-preserving schemes are derived from the stochastic contact Hamilton-Jacobi equation. These numerical tests are conducted to show the effectiveness of the created method by means of the simulations of its orbits on a long time interval and comparison with  other two methods. The results demonstrate  that the method is valid and the numerical experiments are performed well. How to construct  various numerical methods with higher order via the stochastic contact Hamilton-Jacobi equation would be our next topic.

\section*{Statements}
All data in this manuscript is available. And all programs will be available on the WEB GitHub.

\section*{Acknowledgments}
This work is supported by NSFC(12141107). It is also supported by the Fund of Fujian Agriculture and Forestry University(111422138). Prof. X. Li was partially supported by the grant DOE DE-SC00222766.

\end{document}